\documentclass{amsart}
\usepackage[normalem]{ulem}
\usepackage[all]{xy}
\usepackage{verbatim}
\usepackage{color}
\usepackage{amsthm}
\usepackage{amsmath}
\usepackage{amssymb}
\usepackage[colorlinks=true]{hyperref}
\numberwithin{equation}{section}

\newtheorem{theorem}[equation]{Theorem}
\newtheorem*{theorem*}{Theorem} \newtheorem{lemma}[equation]{Lemma}

\newtheorem*{conjecture*}{Mamma Conjecture}
\newtheorem*{conjecture1*}{Mamma Conjecture (revisited)}
\newtheorem{proposition}[equation]{Proposition}

\newtheorem*{corollary*}{Corollary}

\theoremstyle{remark}

\newtheorem{example}[equation]{Example}

\theoremstyle{remark}
\newtheorem{remark}[equation]{Remark}

\newcommand{\cA}{{\mathcal A}}
\newcommand{\cB}{{\mathcal B}}
\newcommand{\cC}{{\mathcal C}}
\newcommand{\cD}{{\mathcal D}}
\newcommand{\cE}{{\mathcal E}}
\newcommand{\cF}{{\mathcal F}}
\newcommand{\cG}{{\mathcal G}}
\newcommand{\cH}{{\mathcal H}}

\newcommand{\cN}{{\mathcal N}}
\newcommand{\cO}{{\mathcal O}}

\newcommand{\Spt}{\mathrm{Spt}}% Spectra

\newcommand{\bbG}{\mathbb{G}}

\newcommand{\bbN}{\mathbb{N}}

\newcommand{\bbZ}{\mathbb{Z}}

\DeclareMathOperator{\id}{id}

\DeclareMathOperator{\Fun}{Fun} % Functor category
\newcommand{\dgcat}{\mathrm{dgcat}} % codimension 
\newcommand{\perf}{\mathrm{perf}}

\newcommand{\dg}{\mathrm{dg}}

\newcommand{\Hom}{\mathrm{Hom}}
\newcommand{\End}{\mathrm{End}}

\newcommand{\rep}{\mathrm{rep}}

\newcommand{\dgHo}{\mathrm{H}^0}

\newcommand{\Ho}{\mathrm{Ho}}

\newcommand{\Hmo}{\mathrm{Hmo}}% Morita homotopy theory
\newcommand{\op}{\mathrm{op}}

\newcommand{\too}{\longrightarrow}

\newcommand{\ie}{\textsl{i.e.}\ }

\newcommand{\U}{\mathrm{U}}

\newcommand{\UU}{\mathbb{U}}

\let\oldmarginpar\marginpar
\def\marginpar#1{\oldmarginpar{\tiny #1}}

\begin{document}

\title[Motivic concentration theorem]{Motivic concentration theorem}
\author{Gon{\c c}alo~Tabuada and Michel Van den Bergh}

\address{Gon{\c c}alo Tabuada, Department of Mathematics, MIT, Cambridge, MA 02139, USA}
\email{tabuada@math.mit.edu}
\urladdr{http://math.mit.edu/~tabuada}
\thanks{G.~Tabuada was partially supported by a NSF CAREER Award.}
\thanks{M.~Van den Bergh is a senior researcher at the Research Foundation -- Flanders}

\address{Michel Van den Bergh, Department of Mathematics, Universiteit Hasselt, 3590 Diepenbeek, Belgium}
\email{michel.vandenbergh@uhasselt.be}
\urladdr{http://hardy.uhasselt.be/personal/vdbergh/Members/~michelid.html}

\date{\today}

\abstract{In this short article, given a smooth diagonalizable group scheme $G$ of finite type acting on a smooth quasi-compact quasi-separated scheme $X$, we prove that (after inverting some elements of representation ring of $G$) all the information concerning the additive invariants of the quotient stack $[X/G]$ is ``concentrated'' in the subscheme of $G$-fixed points $X^G$. Moreover, in the particular case where $G$ is connected and the action has finite stabilizers, we compute the additive invariants of $[X/G]$ using solely the subgroups of roots of unity of $G$. As an application, we establish a Lefschtez-Riemann-Roch formula for the computation of the additive invariants of proper push-forwards.}}

\maketitle
%\vskip-\baselineskip

%-------------------------------------------------------------------------------
\section{Introduction and statement of results}\label{sec:intro}
%-------------------------------------------------------------------------------
A {\em dg category $\cA$}, over a base field
$k$ (of characteristic $p\geq 0$), is a category enriched over complexes of $k$-vector spaces; see
\S\ref{sub:dg}. Every (dg) $k$-algebra
$A$ gives naturally rise to a dg category with a single
object. Another source of examples is provided by schemes (or, more generally, by algebraic
stacks) since the
category of perfect complexes $\perf(X)$ of every
quasi-compact quasi-separated $k$-scheme $X$ (or algebraic
stack) admits a canonical dg enhancement $\perf_\dg(X)$; see \S\ref{sub:perfect}. Let us denote by $\dgcat(k)$ the category of (essentially small) dg categories.

An {\em additive invariant} is a functor $E\colon \dgcat(k) \to \mathrm{D}$, with values in an additive category, which inverts Morita equivalences and sends semi-orthogonal decompositions in the sense of Bondal-Orlov \cite{BO} to direct sums; see \S\ref{sec:additive}. Examples of additive invariants include algebraic $K$-theory and its variants, cyclic homology and its variants, topological Hochschild homology and its variants, etc. Given a $k$-scheme $X$ (or algebraic stack) as above, we will often write $E(X)$ instead of $E(\perf_\dg(X))$.

Let $G$ be a smooth diagonalizable group $k$-scheme of finite type and $X$ a smooth quasi-compact quasi-separated $k$-scheme $X$ equipped with a $G$-action. In what follows, we will write $[X/G]$ for the associated (global) quotient stack, $G^\vee$ for the group of characters of $G$, and $R(G)\simeq \bbZ[G^\vee]$ for the representation ring of $G$. 

As explained below in \S\ref{sec:action}, given an additive invariant $E\colon \dgcat(k) \to \mathrm{D}$, the Grothendieck ring $K_0([X/G])$, \ie the $G$-equivariant Grothendieck ring of $X$, acts naturally on the object $E([X/G])\in \mathrm{D}$. By pre-composing this action with the ring homomorphism $R(G) \to K_0([X/G])$ (induced by pull-back along the projection map $X\to\bullet:=\mathrm{Spec}(k)$), we hence obtain an action of $R(G)$ on $E([X/G])$. Given a multiplicative set $S \subset R(G)$, consider the following presheaf of abelian groups:
\begin{equation}\label{eq:presheaf}
S^{-1}E([X/G]):= \Hom_{\mathrm{D}}(-,E([X/G]))\otimes_{R(G)} S^{-1}R(G)\,.
\end{equation}
Note that since $S^{-1}R(G)$ can be written as a filtered colimit of free finite $R(G)$-modules, the presheaf \eqref{eq:presheaf} belongs to the category $\mathrm{Ind}(\mathrm{D})$ of ind-objects\footnote{For the general theory of ind-objects, we invite the reader to consult \cite{SGA4,AM}.} in $\mathrm{D}$.

Let $H$ be a closed diagonalizable subgroup of $G$. In what follows, we will write $X^H$ for the smooth subscheme of $H$-fixed points and $S_H$ for the multiplicative set generated by the elements $(1 - \chi) \in R(G)\simeq \bbZ[G^\vee]$, where $\chi \in G^\vee$ is any character of $G$ whose restriction to $H$ is non-trivial.% \marginpar{\Michel{[Michel: do you know how to characterize the multiplicative set $S_H$?]}}

Under the above notations and assumptions, our first main result is the following:
\begin{theorem}[Motivic concentration]\label{thm:localization}
We have an isomorphism of ind-objects
$$
E(\iota^\ast)\colon S^{-1}_H E([X/G]) \stackrel{\simeq}{\too} S^{-1}_H E([X^H/G])
$$
induced by pull-back along the closed immersion $\iota\colon X^H \hookrightarrow X$. 

Moreover, its inverse is given by the following composition
$$
S^{-1}_H E([X^H/G]) \stackrel{(\lambda_{-1}(\cN)\cdot -)^{-1}}{\too} S^{-1}_HE([X^H/G]) \stackrel{E(\iota_\ast)}{\too} S^{-1}_H E([X/G])\,,
$$
where $\cN$ stands for the conormal bundle of the closed immersion $\iota\colon X^H \hookrightarrow X$, $\lambda_{-1}(\cN)$ for the Grothendieck class $\sum_j (-1)^j [\bigwedge^j(\cN)] \in K_0([X^H/G])$, and $-\cdot -$ for the induced action of the ring $K_0([X^H/G])$ on the ind-object $S^{-1}_H E([X^H/G])$. 
\end{theorem}
Intuitively speaking, Theorem \ref{thm:localization} shows that (after inverting the multiplicative set $S_H$) all the information concerning the additive invariants of the quotient stack $[X/G]$ is ``concentrated'' in the quotient stack $[X^H/G]$. Since Theorem \ref{thm:localization} holds for {\em every} additive invariant, we named it the ``motivic concentration theorem''.
\begin{remark}[Generalization]
Let $\cH$ be a flat quasi-coherent sheaf of algebras over $[X/G]$, \ie a $G$-equivariant flat quasi-coherent sheaf of algebras over $X$. Given an additive invariant $E\colon \dgcat(k) \to \mathrm{D}$, let us write $E([X/G];\cH)$ for the object $E(\perf_\dg([X/G];\cH)) \in \mathrm{D}$, where $\perf_\dg([X/G];\cH)$ stands for the canonical dg enhancement of the category of $G$-equivariant perfect $\cH$-modules $\perf([X/G];\cH)$. As explained in Remark \ref{rk:generalization}, Theorem \ref{thm:localization} holds more generally with $S^{-1}_HE([X/G])$ and $S^{-1}_HE([X^H/G])$ replaced by $S^{-1}_HE([X/G];\cH)$ and $S^{-1}_HE([X^H/G];\iota^\ast(\cH))$.
\end{remark}
\begin{remark}[Localization at prime ideals]\label{rk:prime}
Let $\rho$ be a prime ideal of the representation ring $R(G)\simeq \bbZ[G^\vee]$. Recall that $G\simeq D(G^\vee)$, where $D(-)$ stands for the classical diagonalizable group scheme construction. On the one hand, similarly to \eqref{eq:presheaf}, we can consider the following presheaf of abelian groups:
$$ E([X/G])_{(\rho)}:=\Hom_{\mathrm{D}}(-,E([X/G]))\otimes_{R(G)}R(G)_{(\rho)}\,.$$
On the other hand, following Segal \cite[Prop.~3.7]{Segal}, we can consider the closed diagonalizable subgroup $H_\rho:=D(G^\vee/K_\rho)$ of $G$ (called the {\em support} of $\rho$), where $K_\rho:=\{\chi \in G^\vee\,|\, 1 - \chi \in \rho \subset \bbZ[G^\vee]\}$. Note that $S_{H_\rho} \cap \rho = \emptyset$ and that $H_\rho$ is maximal with respect to this property. Therefore, by further inverting the elements $R(G)\backslash (S_{H_\rho} \cup \rho)$, we conclude that Theorem \ref{thm:localization} holds similarly with $S^{-1}_HE([X/G])$ and $S^{-1}_H E([X^H/G])$ replaced by $E([X/G])_{(\rho)}$ and $E([X^{H_\rho}/G])_{(\rho)}$, respectively.
\end{remark}
Given an additive category $\mathrm{D}$, let us write $- \otimes_\bbZ -$ for the canonical action of the category of finite free $\bbZ$-modules $\mathrm{free}(\bbZ)$ on $\mathrm{D}$. This action extends naturally to an action of $\mathrm{Ind}(\mathrm{free}(\bbZ))$ on $\mathrm{Ind}(\mathrm{D})$. Our second main result is the following:
\begin{theorem}\label{thm:roots}
Assume that the base field $k$ (of characteristic $p\geq 0$) contains the $l^{\mathrm{th}}$ roots of unity for every~prime~$l \neq p$ such that $(G^\vee)_{l\text{-}\text{torsion}}\neq 0$. Under this assumption, we have an isomorphism of ind-objects:
$$
S^{-1}_{G} E([X/G]) \simeq E(X^G) \otimes_\bbZ S^{-1}_G R(G)\,.
$$
\end{theorem}
Note that when $G$ is moreover connected, \ie a $k$-split torus $T$, the assumption of Theorem \ref{thm:roots} is vacuous. In this case, we have an isomorphism of ind-objects 
$$S^{-1}_{T} E([X/T]) \simeq E(X^T) \otimes_\bbZ \bbZ[t_1^\pm, \ldots, t_r^{\pm}][\{(1-t^j_i)^{-1}\}_{i, j}])\,,$$
where $r$ stands for the rank of $T$, $1\leq i \leq r$, and $j\neq 0 \in \bbZ$.

Similarly to Theorem \ref{thm:localization}, Theorem \ref{thm:roots} shows that (after inverting the multiplicative set $S_G$) all the information concerning the additive invariants of the quotient stack $[X/G]$ is ``concentrated'' in the subscheme of $G$-fixed points $X^G$.

\medskip

We now illustrate Theorems \ref{thm:localization} and \ref{thm:roots} in several examples:
\begin{example}[Algebraic $K$-theory]\label{ex:K-theory1}
Algebraic $K$-theory gives rise to an additive invariant $K\colon \dgcat(k) \to \Ho(\Spt)$ with values in the category of spectra; see \cite[\S2.2.1]{book}. Hence, Theorem \ref{thm:localization} applied to $E=K$ yields an isomorphism of ind-objects:
\begin{equation}\label{eq:iso-spectra}
K(\iota^\ast)\colon S^{-1}_HK([X/G]) \stackrel{\simeq}{\too} S^{-1}_H K([X^H/G])\,.
\end{equation}
Consequently, we obtain, in particular, isomorphisms of abelian groups:
\begin{equation}\label{eq:inverse1}
K_\ast(\iota^\ast)\colon S^{-1}_H K_\ast([X/G])  \stackrel{\simeq}{\too} S^{-1}_H K_\ast([X^H/G])\,.
\end{equation}
Several variants of algebraic $K$-theory such as Karoubi-Villamayor $K$-theory, homotopy $K$-theory, and \'etale $K$-theory, are also additive invariants; see \cite[\S2.2.2-\S2.2.6]{book}. Hence, isomorphisms similar to \eqref{eq:iso-spectra}-\eqref{eq:inverse1} also hold for all these variants.

The above isomorphisms \eqref{eq:inverse1} and their explicit inverses, with $S^{-1}_H K_\ast([X/G])$ and $S^{-1}_H K_\ast([X^H/G])$ replaced by $K_\ast([X/G])_{(\rho)}$ and $K_\ast([X^{H_\rho}/G])_{(\rho)}$ (see Remark \ref{rk:prime}), were originally established by Thomason in \cite[Thm.~2.1 and Prop.~3.1]{Thomason} under the weaker assumption that $X$ is a regular algebraic space. Previously, in the particular case of the Grothendieck group, the  isomorphism \eqref{eq:inverse1} and its explicit inverse were established by Nielsen in \cite[Thm.~3.2]{Nielsen} under the stronger assumptions that $X$ is a smooth projective $k$-scheme and that $k$~is~algebraically~closed.  
\end{example}
\begin{example}[Mixed complex]\label{ex:mixed1}
Recall from Kassel \cite[\S1]{Kassel} that a {\em mixed complex} is a (right) dg module over the algebra of
  dual numbers $\Lambda:=k[\epsilon]/\epsilon^2$, where
  $\mathrm{deg}(\epsilon)=-1$ and $d(\epsilon)=0$. The mixed complex construction gives rise to an additive invariant
  $\mathrm{C}\colon \dgcat(k) \to \cD(\Lambda)$ with values in the derived category of $\Lambda$; see \cite[\S2.2.7]{book}. Hence, Theorem \ref{thm:localization} applied to $E=\mathrm{C}$ yields an isomorphism of ind-objects:
\begin{equation}\label{eq:iso-mixed}
\mathrm{C}(\iota^\ast)\colon S^{-1}_H\mathrm{C}([X/G]) \stackrel{\simeq}{\too} S^{-1}_H \mathrm{C}([X^H/G])\,.
\end{equation}
Cyclic homology and all its variants such as Hochschild homology, negative cyclic homology, and periodic cyclic homology, factor through $\mathrm{C}$. Consequently, an isomorphism similar to \eqref{eq:iso-mixed} also holds for all these invariants. To the best of the authors' knowledge, all these isomorphisms are new in the literature.
\end{example}
\begin{example}[Periodic cyclic homology]\label{ex:HP}
Assume that $\mathrm{char}(k)=0$. Periodic cyclic homology gives rise to an additive invariant $HP_\pm\colon \dgcat(k) \to \mathrm{Vect}_{\bbZ/2}(k)$ with values in the category of $\bbZ/2$-graded $k$-vector spaces; see \cite[\S2.2.11]{book}. Moreover, thanks to the Hochschild-Kostant-Rosenberg theorem, we have an isomorphism $HP_\pm(Y)\simeq (\bigoplus_{i \,\mathrm{even}}H^i_{\mathrm{dR}}(Y), \bigoplus_{i \, \mathrm{odd}}H^i_{\mathrm{dR}}(Y))$ for every smooth $k$-scheme $Y$, where $H^\ast_{\mathrm{dR}}(-)$ stands for de Rham cohomology. Therefore, Theorem \ref{thm:roots} applied to $E=HP_\pm$ yields, in particular, an isomorphism of $\bbZ/2$-graded $k$-vector spaces:
$$
S^{-1}_G HP_\pm([X/G])\simeq (\bigoplus_{i \, \mathrm{even}}H^i_{\mathrm{dR}}(X^G), \bigoplus_{i \, \mathrm{odd}}H^i_{\mathrm{dR}}(X^G))\otimes_\bbZ S^{-1}_G R(G)\,.
$$
This description of the periodic cyclic homology of the quotient stack $[X/G]$ in terms of the de Rham cohomology of the subscheme of $G$-fixed points $X^G$ is, to the best of the authors' knowledge, new in the literature.
\end{example}
\begin{example}[Topological Hochschild homology]\label{ex:THH}
Topological Hochschild homology gives rise to an additive invariant $THH\colon \dgcat(k) \to \Ho(\Spt)$; see \cite[\S2.2.12]{book}. Hence, Theorem \ref{thm:localization} applied to $E=THH$ yields an isomorphism of ind-objects:
\begin{equation}\label{eq:iso-THH}
THH(\iota^\ast)\colon S^{-1}_HTHH([X/G]) \stackrel{\simeq}{\too} S^{-1}_H THH([X^H/G])\,.
\end{equation}
Topological Hochschild homology and all its variants such as topological cyclic homology, topological negative cyclic homology, and topological periodic cyclic homology, are also additive invariants; consult \cite{Lars,NS}\cite[\S2.2.13]{book}. Consequently, an isomorphism similar to \eqref{eq:iso-THH} also holds for all these variants. To the best of the authors' knowledge, all these isomorphisms are new in the literature.
\end{example}
\begin{example}[Topological periodic cyclic homology]\label{ex:TP}
Assume that $k$ is a perfect field of characteristic $p>0$. Let $W(k)$ be the ring of $p$-typical Witt vectors of $k$ and $K:= W(k)_{1/p}$ the fraction field of $W(k)$. Periodic cyclic homology gives rise to an additive invariant $TP_\pm(-)_{1/p}\colon \dgcat(k) \to \mathrm{Vect}_{\bbZ/2}(K)$ with values in the category of $\bbZ/2$-graded $K$-vector spaces; see \cite[Thm.~2.3]{finite}. Moreover, following Scholze (see \cite[Thm.~5.2]{CD}), we have $TP_\pm(Y)_{1/p} \simeq (\bigoplus_{i \, \mathrm{even}}H^i_{\mathrm{crys}}(Y), \bigoplus_{i \, \mathrm{odd}}H^i_{\mathrm{crys}}(Y))$ for every smooth proper $k$-scheme $Y$, where $H^\ast_{\mathrm{crys}}(-)$ stands for crystalline cohomology. Therefore, Theorem \ref{thm:roots} applied to $E=TP_\pm(-)_{1/p}$ yields, in particular, an isomorphism of $\bbZ/2$-graded $K$-vector spaces:
$$
S^{-1}_G TP_\pm([X/G])_{1/p}\simeq (\bigoplus_{i \, \mathrm{even}}H^i_{\mathrm{crys}}(X^G), \bigoplus_{i \, \mathrm{odd}}H^i_{\mathrm{crys}}(X^G))\otimes_\bbZ S^{-1}_G R(G)\,.
$$
Similarly to the above Example \ref{ex:HP}, this description of the topological periodic cyclic homology of the quotient stack $[X/G]$ in terms of the crystalline cohomology of the subscheme of $G$-fixed points $X^G$ is new in the literature.
\end{example}
%-------------------------------------------------------------------------------
\subsection*{Proper push-forwards}
%-------------------------------------------------------------------------------
The following result is an immediate application of the above Theorems \ref{thm:localization} and \ref{thm:roots}:
\begin{theorem}[Motivic Lefschetz-Riemann-Roch formula]\label{thm:push-forward}
Given a $G$-equivariant proper map $f\colon X \to Y$, between smooth quasi-compact quasi-separated $k$-schemes, we have the following commutative diagram of ind-objects:
\begin{equation}\label{eq:Euler}
\xymatrix{
S^{-1}_H E([X/G]) \ar[dd]_-{E(f_\ast)} \ar[rr]^{E(\iota^\ast)} && S^{-1}_H E([X^H/G]) \ar[d]^-{(\lambda_{-1}(\cN)\cdot -)^{-1}} \\
&& S^{-1}_H E([X^H/G]) \ar[d]^-{E((f\circ \iota)_\ast)} \\
S^{-1}_H E([Y/G]) \ar@{=}[rr] && S^{-1}_H E([Y/G])\,.
}
\end{equation}
Moreover, in the particular case where $X^G$ consists of a finite set of closed points and $Y=\bullet$, the commutative diagram \eqref{eq:Euler} (with $H=G$) reduces to the following commutative diagram of ind-objects\footnote{In the particular case where $X^H$ consists of a finite set of closed points and $G$ is moreover connected, the $G$-action on $X^H$ is necessarily trivial. Consequently, in this case, the above diagram \eqref{eq:Euler1} holds similarly with $S^{-1}_G E(-)$ and $X^G$ replaced by $S^{-1}_H E(-)$ and $X^H$, respectively.}
\begin{equation}\label{eq:Euler1}
\xymatrix{
S^{-1}_G E([X/G]) \ar[dd]_-{E(f_\ast)} \ar[rr]^{E(\iota^\ast)} && \bigoplus_{x \in X^G} S^{-1}_G E([\bullet/G]) \ar[d]^-{\bigoplus_{x \in X^G}(\lambda_{-1}(T^\vee_x)\cdot -)^{-1}} \\
&& \bigoplus_{x \in X^G} S^{-1}_G E([\bullet/G]) \ar[d]^-{\nabla} \\
S^{-1}_G E([\bullet/G]) \ar@{=}[rr] && S^{-1}_G E([\bullet/G])\,,
}
\end{equation}
where $\nabla$ stands for the co-diagonal map and $T^\vee_x$ for the dual of the tangent bundle of $X$ at the point $x$. Furthermore, whenever $k$ contains the $l^{\mathrm{th}}$ roots of unity for every prime $l\neq p$ such that $(G^\vee)_{l\text{-}\text{torsion}}\neq 0$, the ind-object $\bigoplus_{x \in X^G} S^{-1}_G E([\bullet/G])$ in \eqref{eq:Euler1} can be replaced by the ind-object $E(k) \otimes_\bbZ \bigoplus_{x \in X^G} S^{-1}_G R(G)$.
\end{theorem}
Intuitively speaking, the commutative diagram \eqref{eq:Euler}, resp. \eqref{eq:Euler1}, shows that after inverting the multiplicative set $S_H$, resp. $S_G$, all the information concerning the additive invariants of the push-forward along $f$, resp. along $X \to \bullet$, is ``concentrated'' in the quotient stack $[X^H/G]$, resp. in the finite set~of~closed~points~$X^G$.

To the best of the authors' knowledge, Theorem \ref{thm:push-forward} is new in the literature. In the particular case where $E=K_0(-)$, the diagram \eqref{eq:Euler1} was originally established by Nielsen in \cite[Prop.~4.5]{Nielsen} (under the stronger assumptions that $X$ is a smooth projective $k$-scheme and that $k$ is algebraically closed) and later by Thomason in \cite[Thm.~3.5]{Thomason} (with $S^{-1}_G K_0(-)$ replaced by $K_0(-)_{(\{0\})}$ (see Remark \ref{rk:prime}) under the weaker assumption that $X$ is a regular algebraic space). Note that in this particular case, the diagram \eqref{eq:Euler1} reduces to the classical Lefschetz-Riemann-Roch formula
\begin{equation}\label{eq:LRR}
\sum_i (-1)^i [H^i(X;\cF)] = \sum_{x \in X^G} \frac{[\cF_x]}{\sum_j (-1)^j [\bigwedge^j (T^\vee_x)]}\quad \mathrm{in}\,\, S^{-1}_GR(G)\,,
\end{equation}
which computes the $G$-equivariant Euler characteristic of every $G$-equivariant perfect complex of $\cO_X$-modules $\cF$ in terms of the finite set of closed points $X^G$. It is well-known that the formula \eqref{eq:LRR} implies many other classical formulas such as the Woods Hole fixed-point formula (see \cite{Woods}), the Weyl's character formula (see \cite{Demazure,Weyl}), the Brion's counting formula (see \cite{Brion1}), etc.
%-------------------------------------------------------------------------------
\subsection*{Torus actions with finite stabilizers}
%-------------------------------------------------------------------------------
In this subsection we assume that $G$ is moreover connected, i.e. a $k$-split torus $T$, and that the $T$-action on $X$ has finite (geometric) stabilizers. Let us denote by $\cC(T)$ the set of all those subgroups $\mu_n \subset T$ such that $X^{\mu_n}\neq \emptyset$. Note that since the $T$-action on $X$ has finite stabilizers, the set $\{n \in \bbN\,|\, \mu_n \in \cC(T)\}$ is finite; in what follows, we will write $r$ for the least common multiple of the elements of this latter set.

Given a subgroup $\mu_n \in \cC(T)$, let $\underline{S}_{\mu_n} \subset R(T)_{1/r}$ be the multiplicative set defined as the pre-image of $1$ under the following $\bbZ[1/r]$-algebra homomorphism %\marginpar{\Michel{[Michel: do you know how to characterize the multiplicative set $\underline{S}_{\mu_n}$?]}}
$$R(T)_{1/r} \stackrel{\text{(a)}}{\too} R(\mu_n)_{1/r} \stackrel{\text{(b)}}{\simeq} \frac{\bbZ[1/r][t]}{\langle t^n -1 \rangle} \simeq \prod_{d|n} \frac{\bbZ[1/r][t]}{\langle \Phi_d(t) \rangle} \stackrel{\text{(c)}}{\too} \frac{\bbZ[1/r][t]}{\langle\Phi_n(t)\rangle}\,,$$
where (a) is the restriction homomorphism, (b) is induced by the choice of a(ny) generator $t$ of the character group $\mu_n^\vee$, $\Phi_d(t)$ stands for the $d^{\mathrm{th}}$ cyclotomic polynomial, and (c) is the projection homomorphism. Under these notations and assumptions, our third main result is the following:
\begin{theorem}\label{thm:stabilizers}
For every additive invariant $E\colon \dgcat(k) \to \mathrm{D}$, with values in a $\bbZ[1/r]$-linear category, we have an isomorphism of ind-objects
\begin{equation}\label{eq:stabilizers}
E([X/T]) \stackrel{\simeq}{\too} \bigoplus_{\mu_n \in \cC(T)} \underline{S}^{-1}_{\mu_n} E([X^{\mu_n}/T])
\end{equation}
induced by pull-back along the closed immersions $X^{\mu_n} \hookrightarrow X$. Moreover, the direct sum on the right-hand side is finite.
\end{theorem}
Intuitively speaking, Theorem \ref{thm:stabilizers} shows that all the information concerning the additive invariants of the quotient stack $[X/T]$ (no invertion is needed!) is ``concentrated'' in the quotient stacks $[X^{\mu_n}/T]$. 

Thanks to Theorem \ref{thm:stabilizers}, the above isomorphism \eqref{eq:stabilizers} holds for algebraic $K$-theory and all its variants, for cyclic homology and all its variants, for topological Hochschild homology and all its variants, etc. In the particular case of algebraic $K$-theory such an isomorphism was originally established by Vezzosi-Vistoli in \cite[\S3]{VV} under the weaker assumption that $X$ is a regular algebraic space. The remaining isomorphisms are, to the best of the authors' knowledge, new in the literature. 

%-------------------------------------------------------------------------------
\subsection*{Proofs}
%-------------------------------------------------------------------------------
Our proof of Theorem \ref{thm:localization}, resp. Theorem \ref{thm:stabilizers}, is different from the proof of Thomason, resp. of Vezzosi-Vistoli, in algebraic $K$-theory. Nevertheless, we do borrow some ingredients from their proofs. In fact, using a certain category of subschemes of the quotient stack $[X/G]$ (see \S\ref{sec:subschemes}), we are able to ultimately reduce the proof of Theorem \ref{thm:localization}, resp. Theorem \ref{thm:stabilizers}, to the proof of the $K_0$-case of Thomason's result, resp. of Vezzosi-Vistoli's result; consult \S\ref{sec:proof1} for details. Note, however, that
we {\em cannot} mimic Thomason's arguments, resp. Vezzosi-Vistoli's arguments, because they depend in an essential way
on the d\'evissage property of $G$-theory (=$K$-theory for smooth schemes), which does {\em not} hold for many
additive invariants. For example, as explained by Keller in \cite[Example 1.11]{Exact}, Hochschild homology, and consequently the mixed
  complex, do {\em not} satisfy d\'evissage.
%-------------------------------------------------------------------------------
\section{Preliminaries}
%-------------------------------------------------------------------------------
Throughout the article, $k$ will be a base field of characteristic $p\geq 0$.
%-------------------------------------------------------------------------------
\subsection{Dg categories}\label{sub:dg}
%-------------------------------------------------------------------------------
Let $(\cC(k),\otimes, k)$ be the category of (cochain) complexes of 
$k$-vector spaces. A {\em dg category $\cA$} is a category enriched over $\cC(k)$
and a {\em dg functor} $F\colon\cA\to \cB$ is a functor enriched over
$\cC(k)$; consult Keller's survey \cite{ICM-Keller}. Recall from \S\ref{sec:intro} that $\dgcat(k)$ stands for the category of (essentially small) dg categories.

Let $\cA$ be a dg category. The opposite dg category $\cA^\op$ has the
same objects and $\cA^\op(x,y):=\cA(y,x)$. The category $\dgHo(\cA)$ has the same objects as $\cA$ and morphisms $\dgHo(\cA)(x,y):=H^0 \cA(x,y)$, where $H^0(-)$ stands for the $0^{\mathrm{th}}$-cohomology functor. A {\em right dg
  $\cA$-module} is a dg functor $M\colon \cA^\op \to \cC_\dg(k)$ with values
in the dg category $\cC_\dg(k)$ of complexes of $k$-vector spaces. Let
us write $\cC(\cA)$ for the category of right dg
$\cA$-modules. Following \cite[\S3.2]{ICM-Keller}, the derived
category $\cD(\cA)$ of $\cA$ is defined as the localization of
$\cC(\cA)$ with respect to the objectwise quasi-isomorphisms. 

A dg functor $F\colon\cA\to \cB$ is called a {\em Morita equivalence} if it
induces an equivalence on derived categories $\cD(\cA) \stackrel{\simeq}{\to}
\cD(\cB)$; see \cite[\S4.6]{ICM-Keller}. As explained in
\cite[\S1.6]{book}, the category $\dgcat(k)$ admits a Quillen model
structure whose weak equivalences are the Morita equivalences. Let us
denote by $\Hmo(k)$ the associated homotopy category.

The {\em tensor product $\cA\otimes\cB$} of dg categories is defined
as follows: the set of objects is the cartesian product of the sets of objects of $\cA$ and $\cB$ and
$(\cA\otimes\cB)((x,w),(y,z)):= \cA(x,y) \otimes \cB(w,z)$. As
explained in \cite[\S2.3]{ICM-Keller}, this construction gives rise to
a symmetric monoidal structure on $\dgcat(k)$, which descends to $\Hmo(k)$.

A {\em dg $\cA\text{-}\cB$-bimodule} is a dg functor
$\mathrm{B}\colon \cA\otimes \cB^\op \to \cC_\dg(k)$ or, equivalently, a
right dg $(\cA^\op \otimes \cB)$-module. A standard example is the dg
$\cA\text{-}\cB$-bimodule
\begin{eqnarray}\label{eq:bimodule2}
{}_F\mathrm{B}:\cA\otimes \cB^\op \too \cC_\dg(k) && (x,z) \mapsto \cB(z,F(x))
\end{eqnarray}
associated to a dg functor $F:\cA\to \cB$. 
%------------------------------------------
\subsection{Additive invariants}\label{sec:additive}
%------------------------------------------
A functor $E\colon \dgcat(k)
\to \mathrm{D}$, with values in an additive category, is called an
{\em additive invariant} if it satisfies the following two conditions:
\begin{itemize}
\item[(i)] It sends the Morita equivalences (see \S\ref{sub:dg}) to isomorphisms.
\item[(ii)] Let $\cA \subseteq \cB$ and $\cC\subseteq \cB$ be dg categories inducing a semi-orthogonal decomposition $\dgHo(\cB)=\langle\dgHo(\cA), \dgHo(\cC) \rangle$ in the sense of Bondal-Orlov \cite{BO}. In this case, the inclusions $\cA \subseteq \cB$ and $\cC\subseteq \cB$ induce an isomorphism $E(\cA) \oplus E(\cC) \stackrel{\simeq}{\to} E(\cB)$.
\end{itemize}
Given small dg categories $\cA$ and $\cB$, let us write $\rep(\cA,\cB)$ for the full triangulated subcategory of $\cD(\cA^\op \otimes \cB)$ consisting of those dg $\cA\text{-}\cB$-modules $\mathrm{B}$ such that for every object $x \in \cA$ the associated right dg $\cB$-module $\mathrm{B}(x,-)\in \cD(\cB)$ belongs to the full triangulated subcategory of compact objects $\cD_c(\cB)$. As explained in \cite[\S1.6.3]{book}, there is a
  natural bijection between $\Hom_{\Hmo(k)}(\cA,\cB)$ and the set of
  isomorphism classes of the category $\rep(\cA,\cB)$. Moreover, under this bijection, the composition law of $\Hmo(k)$ corresponds to the (derived) tensor product of bimodules. 

The {\em additivization} $\Hmo_0(k)$ of $\Hmo(k)$ is defined as the category with the same objects as $\Hmo(k)$ and morphisms $\Hom_{\Hmo_0(k)}(\cA,\cB):=K_0\rep(\cA,\cB)$. Since the dg $\cA\text{-}\cB$-bimodules
  \eqref{eq:bimodule2} belong to $\rep(\cA,\cB)$, we have the symmetric monoidal functor:
\begin{eqnarray}\label{eq:universal}
\U\colon \dgcat(k) \too \Hmo_0(k) & \cA \mapsto \cA & (\cA \stackrel{F}{\to} \cB) \mapsto [{}_F\mathrm{B}]\,.
\end{eqnarray}
As explained in
  \cite[\S2.3]{book}, this functor is the {\em universal} additive invariant, \ie given any additive category $\mathrm{D}$, pre-composition with $\U$
  gives rise to an equivalence
\begin{eqnarray}\label{eq:equivalence1}
\Fun_{\mathrm{additive}}(\Hmo_0(k),\mathrm{D}) \stackrel{\simeq}{\too} \Fun_{\mathrm{add}}(\dgcat(k),\mathrm{D})\,,
\end{eqnarray}
where the left-hand side stands for the category of additive functors and the right-hand side for the category of additive invariants.
%-------------------------------------------------------------------------------
\subsection{Derived categories of quotient stacks}\label{sub:perfect}
%-------------------------------------------------------------------------------
Let $G$ be an affine group $k$-scheme of finite type and $X$ a quasi-compact quasi-separated $k$-scheme equipped with a $G$-action. Let us denote by $\mathrm{Mod}([X/G])$ the Grothendieck category of $G$-equivariant $\cO_X$-modules and by $\mathrm{Qcoh}([X/G])$, resp. by $\mathrm{coh}([X/G])$, the full subcategory of $G$-equivariant quasi-coherent, resp. coherent, $\cO_X$-modules. In what follows, we will write $\cD([X/G]):=\cD(\mathrm{Mod}([X/G]))$ for the derived category of the quotient stack $[X/G]$, $\cD_{\mathrm{Qcoh}}([X/G])\subset \cD([X/G])$ for the full triangulated subcategory of those complexes of $G$-equivariant $\cO_X$-modules whose cohomology belongs to $\mathrm{Qcoh}([X/G])$, $\cD^b(\mathrm{coh}([X/G]))\subset \cD_{\mathrm{Qcoh}}([X/G])$ for the full triangulated subcategory of bounded complexes of $G$-equivariant coherent $\cO_X$-modules, and $\perf([X/G]) \subset \cD^b(\mathrm{coh}([X/G]))$ for the full triangulated subcategory of perfect complexes of $G$-equivariant $\cO_X$-modules.

Let $\cE x$ be an exact category. As explained in \cite[\S4.4]{ICM-Keller}, the {\em derived dg category $\cD_\dg(\cE x)$ of $\cE x$} is defined as the Drinfeld's dg quotient $\cC_\dg(\cE x)/\cA{c}_\dg(\cE x)$ of the dg category of complexes over $\cE x$ by its full dg subcategory of acyclic complexes. 

Following the above notations, we will write $\cD_\dg([X/G])$ for the dg category $\cD_\dg(\cE x)$, with $\cE x:=\mathrm{Mod}([X/G])$, and $\cD_{\mathrm{Qcoh}, \dg}([X/G])$, $\cD^b_\dg(\mathrm{coh}([X/G]))$, and $\perf_\dg([X/G])$, for the corresponding full dg subcategories. 
\begin{proposition}[Trivial action]\label{prop:trivial}
Assume that the category $\cD_{\mathrm{Qcoh}}([\bullet/G])$ is compactly generated. Under this assumption, whenever the $G$-action on $X$ is trivial, we have the following Morita equivalence:
\begin{eqnarray}\label{eq:MoritaEquivalence}
\perf_{\dg}(X) \otimes \perf_{\dg}([\bullet/G])\too \perf_{\dg}([X/G]) && (\cF, V)\mapsto \cF\boxtimes V\,.
\end{eqnarray}
\end{proposition}
\begin{remark}\label{rk:compactgeneration}
As proved in \cite[Thm.~A]{HR} and \cite[Lem.~4.1]{HNR}, the category $\cD_{\mathrm{Qcoh}}([\bullet/G])$ is compactly generated if and only if $k$ is of characteristic zero or if $k$ is of positive characteristic and $\overline{G}:=G\otimes_k \overline{k}$ does not contains a copy of the additive group $\bbG_a$.
\end{remark}
\begin{proof}
Given any $\cF \in \cD_{\mathrm{Qcoh}}(X)$, any $V \in \cD_{\mathrm{Qcoh}}([\bullet/G])$, and any $\cG \in \cD_{\mathrm{Qcoh}}([X/G])$, we have the following classical tensor-Hom relation:
\begin{equation}\label{eq:relation}
{\bf R}\Hom_{[X/G]}(\cF \boxtimes V, \cG)\simeq {\bf R}\Hom_{[\bullet/G]}(V, {\bf R}\Hom_X(\cF,\cG))\,.
\end{equation}
The relation \eqref{eq:relation} implies that if $\cF$ and $V$ are compact objects, then $\cF\boxtimes V \in \cD_{\mathrm{Qcoh}}([X/G])$ is also a compact object. Moreover, if ${\bf R}\Hom_{[X/G]}(\cF\boxtimes V, \cG)=0$ for every $\cF$ and $V$, then $\cG$ is necessarily equal to zero. Furthermore, if $\cF$ and $V$ are perfect complexes, then $\cF\boxtimes V$ is also a perfect complex. Since the categories $\cD_{\mathrm{Qcoh}}(X)$ and $\cD_{\mathrm{Qcoh}}([\bullet/G])$ are compactly generated by perfect complexes (consult \cite[Thm.~3.1.1]{MRXXX1} and \cite[Thm.~A (b)]{HR}, respectively) the above three facts imply that the category $\cD_{\mathrm{Qcoh}}([X/G])$ is also compactly generated by perfect complexes. Finally, given any two perfect complexes $\cF_1, \cF_2 \in \perf(X)$ and any two $G$-representations $V_1, V_2 \in \perf([\bullet/G])$, note that \eqref{eq:relation} also implies that
$$ {\bf R}\Hom_{[X/G]}(\cF_1 \boxtimes V_1, \cF_2 \boxtimes V_2)\simeq {\bf R}\Hom_{[\bullet/G]}(V_1, V_2) \otimes {\bf R}\Hom_X(\cF_1, \cF_2)\,.$$
This allows us to conclude that the dg functor \eqref{eq:MoritaEquivalence} is a Morita equivalence.
\end{proof}
%---------------------------
\subsection{Action of the Grothendieck ring}\label{sec:action}
%---------------------------
Let $G$ be an affine group $k$-scheme of finite type and $X$ a quasi-compact quasi-separated $k$-scheme equipped with a $G$-action. Since the tensor product $-\otimes_X-$ makes the dg category $\perf_\dg([X/G])$ into a commutative monoid and the universal additive invariant \eqref{eq:universal} is symmetric monoidal, we obtain a commutative monoid $\U([X/G])$ in the category $\Hmo_0(k)$. Making use of the following natural ring isomorphism 
$$ \Hom_{\Hmo_0(k)}(\U(k),\U([X/G])):=K_0\rep(k, \perf_\dg([X/G]))\simeq K_0(\perf([X/G]))\,,$$
we hence conclude that the Grothendieck ring $K_0([X/G])\simeq  K_0(\perf([X/G]))$ acts on the object $\U([X/G])$ (and also that the monoid structure of $\U([X/G])$ is $K_0([X/G])$-linear). Concretely, this action can be explicitly described as follows:
\begin{eqnarray*}
K_0(\perf([X/G]))\too K_0(\rep(\perf_{\dg}([X/G]),\perf_{\dg}([X/G]))) & [\cF]\mapsto [{}_{(\cF\otimes_X -)} \mathrm{B}]\,. &
\end{eqnarray*}
Given any additive invariant $E\colon \dgcat(k) \to \mathrm{D}$, the equivalence of categories \eqref{eq:equivalence1} implies, by functoriality, that $K_0([X/G])$ acts on the object $E([X/G]) \in \mathrm{D}$.
%------------------------------------
\section{Category of subschemes of a quotient stack}
\label{sec:subschemes}
%------------------------------------
Let $G$ be a smooth affine group $k$-scheme of finite type and $X$ a smooth quasi-compact quasi-separated $k$-scheme equipped with a $G$-action. In this section, we construct a certain category\footnote{In the case of a constant finite group $k$-scheme $G$, a related category of $G$-equivariant smooth ``covers'' of $X$ was constructed in \cite[\S5]{Additive}.} $\mathbb{S}\mathrm{ub}^G_0(X)$ of $G$-stable smooth closed subschemes of $X$. This category, which is of independent interest, will play a key role in the proof of Theorems \ref{thm:localization} and \ref{thm:stabilizers}; consult \S\ref{sec:proofs} below.

\subsection*{Definition of the category $\mathbb{S}\mathrm{ub}^G_0(X)$} Let $\mathbb{S}\mathrm{ub}^G(X)$ be the category whose objects are the $G$-stable closed immersions $Y \stackrel{\tau}{\hookrightarrow} X$, with $Y$ a smooth quasi-compact quasi-separated $k$-scheme. In what follows, in order to simplify the exposition, we will often write $Y$. Given two objects $Y_1$ and $Y_2$, $\Hom_{\mathbb{S}\mathrm{ub}^G(X)}(Y_1, Y_2)$ is defined as the set of isomorphism classes of the bounded derived category
\begin{equation}\label{eq:coh-support}
\cD^b(\mathrm{coh}_{Y_1 \times_X Y_2}([(Y_1 \times Y_2)/G]))\subset \perf([(Y_1 \times Y_2)/G])
\end{equation}
of those $G$-equivariant coherent $\cO_{Y_1 \times Y_2}$-modules 
that are (topologically) supported on the closed subscheme $Y_1 \times_X Y_2$. Note that since the quotient stack $[(Y_1 \times Y_2)/G]$ is smooth, every bounded complex of $G$-equivariant coherent $\cO_{Y_1 \times Y_2}$-modules is perfect. Given three objects $Y_1$, $Y_2$, and $Y_3$, the composition law
$$\Hom_{\mathbb{S}\mathrm{ub}^G(X)}(Y_2,Y_3) \times \Hom_{\mathbb{S}\mathrm{ub}^G(X)}(Y_1,Y_2) \too \Hom_{\mathbb{S}\mathrm{ub}^G(X)}(Y_1,Y_3)$$
is induced by the classical (derived) ``pull-back/push-forward'' formula
\begin{equation}\label{eq:assignment}
(\cE_{23},\cE_{12}) \mapsto (q_{13})_\ast ((q_{23})^\ast(\cE_{23}) \otimes^{\bf L} (q_{12})^\ast(\cE_{12})) \,,
\end{equation}
where $q_{ij}$ stands for the projection from the triple fiber product onto its $ij$-factor. Finally, the identity of an object $Y$ is the (isomorphism class of the) $G$-equivariant structure sheaf $\cO_\Delta$ of the diagonal $\Delta$ in $Y \times Y$. 

The {\em additivization} $\mathbb{S}\mathrm{ub}^G_0(X)$ of $\mathbb{S}\mathrm{ub}^G(X)$ is defined by formally adding all finite direct sums to the category which has the same objects as $\mathbb{S}\mathrm{ub}^G(X)$ and morphisms 
$$\Hom_{\mathbb{S}\mathrm{ub}^G_0(X)}(Y_1,Y_2):=K_0(\cD^b(\mathrm{coh}_{Y_1 \times_X Y_2}([(Y_1 \times Y_2)/G])))\,.$$ 
Note that, since the above formula \eqref{eq:assignment} is exact in each one of the variables, the composition law of $\mathbb{S}\mathrm{ub}^G(X)$ extends naturally to $\mathbb{S}\mathrm{ub}^G_0(X)$. Let us denote by $\UU\colon \mathbb{S}\mathrm{ub}^G(X)\to \mathbb{S}\mathrm{ub}^G_0(X)$ the canonical functor. Note also that thanks to Quillen's d\'evissage theorem \cite[\S5]{Quillen} and to the definition of $G$-theory, we have isomorphisms:
$$
\Hom_{\mathbb{S}\mathrm{ub}^G_0(X)}(Y_1, Y_2)\simeq G_0([(Y_1 \times_X Y_2)/G])\,.
$$
In particular, we have ring isomorphisms: 
$$\End_{\mathbb{S}\mathrm{ub}^G_0(X)}(Y)\simeq G_0([Y/G]) \simeq K_0([Y/G])\,.$$
\subsection*{Relation(s) between the categories $\mathbb{S}\mathrm{ub}^G_0(X)$ and $\Hmo_0(k)$}
Given two objects $Y_1$ and $Y_2$ of the category $\mathbb{S}\mathrm{ub}^G(X)$, consider the exact functor
$$ \cD^b(\mathrm{coh}_{Y_1 \times_X Y_2}([(Y_1 \times Y_2)/G]))\too \rep(\perf_\dg([Y_1/G]), \perf_\dg([Y_2/G])$$
that sends a bounded complex of $G$-equivariant coherent $\cO_{Y_1 \times Y_2}$-modules $\cE_{12}$ (supported on the closed subscheme $Y_1 \times_X Y_2$) to the following Fourier-Mukai dg-functor: 
\begin{eqnarray*}
\label{eq:fm}
\Phi_{\cE_{12}}\colon\perf_\dg([Y_1/G])\too \perf_\dg([Y_2/G]) \qquad \cF \mapsto (q_2)_\ast((q_1)^\ast(\cF ) \otimes^{\bf L} \cE_{12})\,.
\end{eqnarray*}
By definition of the categories $\mathbb{S}\mathrm{ub}^G(X)$ and $\Hmo(k)$, the above constructions lead to a well-defined functor
\begin{eqnarray*}
\mathbb{S}\mathrm{ub}^G(X)\too \Hmo(k) & Y \mapsto \perf_\dg([Y/G]) & \cE_{12} \mapsto {}_{\Phi_{\cE_{12}}}\mathrm{B}\,,
\end{eqnarray*}
which naturally extends to the additive categories:
\begin{eqnarray*}
\Psi\colon \mathbb{S}\mathrm{ub}^G_0(X) \too \Hmo_0(k) && \UU(Y) \mapsto \U([Y/G])\,.
\end{eqnarray*}
\begin{remark}[Sheaves of algebras]\label{rk:sheaves}
Let $\cH$ be a flat quasi-coherent sheaf of algebras over $[X/G]$, \ie a $G$-equivariant flat quasi-coherent sheaf of algebras over $X$. Given two objects $Y_1\stackrel{\tau_1}{\hookrightarrow} X$ and $Y_2 \stackrel{\tau_2}{\hookrightarrow} X$ of $\mathbb{S}\mathrm{ub}^G(X)$, consider the exact functor
$$
\cD^b(\mathrm{coh}_{Y_1 \times_X Y_2}([(Y_1 \times Y_2)/G]))\to  \rep(\perf_\dg([Y_1/G];\tau_1^\ast(\cH)), 
\perf_\dg([Y_2/G];\tau_2^\ast(\cH)))
$$
defined, as above, by the assignment $\cE_{12} \mapsto {}_{\Phi_{\cE_{12}}}\mathrm{B}$. This leads to a functor
\begin{eqnarray*}
\mathbb{S}\mathrm{ub}^G(X)\too \Hmo(k) & (Y\stackrel{\tau}{\hookrightarrow} X) \mapsto \perf_\dg([Y/G]; \tau^\ast(\cH)) & \cE_{12} \mapsto {}_{\Phi_{\cE_{12}}} \mathrm{B}\,,
\end{eqnarray*}
which naturally extends to the additive categories:
\begin{eqnarray*}
\Psi_\cH\colon \mathbb{S}\mathrm{ub}^G_0(X) \too \Hmo_0(k) && \UU(Y \stackrel{\tau}{\hookrightarrow} X) \mapsto \U([Y/G]; \tau^\ast(\cH))\,.
\end{eqnarray*}
\end{remark}
\subsection*{Some properties of the category $\mathbb{S}\mathrm{ub}^G_0(X)$ and of the functor $\Psi$} In what follows, we describe three important properties that will be used in the sequel.
\subsubsection{Pull-back and push-forward}\label{rk:pullback}
Let $Y_1\stackrel{\tau_1}{\to} X$ and $Y_2\stackrel{\tau_2}{\to} X$ be two objects of the category $\mathbb{S}\mathrm{ub}^G(X)$. Given a $G$-stable closed immersion $\iota\colon Y_1 \hookrightarrow Y_2$ such that $\tau_2 \circ \iota = \tau_1$, its {\em pull-back} $\UU(\iota^\ast)\colon \UU(Y_2) \to \UU(Y_1)$, resp. {\em push-forward} $\UU(\iota_\ast)\colon \UU(Y_1) \to \UU(Y_2)$, is defined as the Grothendieck class $[(\iota \times_X\id)_\ast(\cO_{Y_1})]$, resp. $[(\id \times_X \iota)_\ast(\cO_{Y_1})]$, of the group $G_0([(Y_2\times_X Y_1)/G])$, resp. $G_0([(Y_1\times_X Y_2)/G])$. Note that $\Psi(\UU(\iota^\ast))=\U(\iota^\ast)$ and $\Psi(\UU(\iota_\ast))=\U(\iota_\ast)$.
\subsubsection{$K_0$-action}\label{rk:K0action}
Let $Y$ be an object of $\mathbb{S}\mathrm{ub}^G(X)$. The push-forward along the diagonal map 
$i_\Delta:Y \hookrightarrow Y\times Y$ leads to an exact functor
\begin{equation}\label{eq:diagonal}
(i_\Delta)_\ast\colon \perf([Y/G])\too  \cD^b(\mathrm{coh}_\Delta([(Y\times Y)/G]))
\end{equation}
that sends the tensor product $-\otimes_Y-$ on the left-hand side to the ``pull-back/push-forward'' formula \eqref{eq:assignment} on the right-hand side. Therefore, by applying $K_0(-)$ to \eqref{eq:diagonal}, we obtain an induced ring morphism $K_0([Y/G])\to \End_{\mathbb{S}\mathrm{ub}_0^G(X)}(\UU(Y))$. In other words, we obtain an action of $K_0([Y/G])$ on the object $\UU(Y)$.
\begin{lemma} \label{lem:compat} The functor $\Psi$ interchanges with the $K_0([Y/G])$-action on $\UU(Y)$ (defined above)  with the $K_0([Y/G])$-action on $\U([Y/G])$ (defined in \S\ref{sec:action}).
\end{lemma}
\begin{proof}
Consider the following commutative diagram:
\begin{equation}\label{eq:diagram}
\xymatrix{
\perf([Y/G])\ar@{=}[d]\ar[rrr]^-{(i_\Delta)_\ast}&&& \cD^b(\mathrm{coh}_\Delta([(Y \times Y)/G]))\ar[d]^{\cE\mapsto  {}_{\Phi_\cE}\mathrm{B}}\\
\perf([Y/G])\ar[rrr]_-{\cF\mapsto{}_{(\cF\otimes_Y-)}\mathrm{B}}&&& \rep(\perf_{\dg}([Y/G]),\perf_{\dg}([Y/G]))\,.
}
\end{equation}
By applying $K_0(-)$ to \eqref{eq:diagram}, we obtain the claimed compatibility.
\end{proof}
\subsubsection{$K_0$-linearity}\label{rk:compatibility} 
Let $Y\stackrel{\tau}{\hookrightarrow} X$ be an object of $\mathbb{S}\mathrm{ub}^G(X)$. By composing the induced ring homomorphism $\tau^\ast\colon K_0([X/G]) \to K_0([Y/G])$ with the $K_0([Y/G])$-action on $\UU(Y)$ described in \S\ref{rk:K0action}, we obtain a $K_0([X/G])$-action on $\UU(Y)$. A simple verification shows that this $K_0([X/G])$-action is compatible with the morphisms of the category $\mathbb{S}\mathrm{ub}_0^G(X)$. This implies that $\mathbb{S}\mathrm{ub}^G_0(X)$ is a $K_0([X/G])$-linear category. Note that since the projection map $X \to \bullet$ induces a ring homomorphism $R(G) \to K_0([X/G])$, the category $\mathbb{S}\mathrm{ub}^G_0(X)$ is also $R(G)$-linear.
%------------------------------------
\section{Proofs}\label{sec:proofs}
%------------------------------------
In this section, making use of the category $\mathbb{S}\mathrm{ub}^G_0(X)$ of $G$-stable smooth closed subschemes of $X$ (consult \S\ref{sec:subschemes}), we prove Theorems \ref{thm:localization}, \ref{thm:roots} and \ref{thm:stabilizers}. 
%---------------------------------------------
\subsection*{Proof of Theorem \ref{thm:localization}}\label{sec:proof1}
%---------------------------------------------
Consider the following morphisms
\begin{eqnarray}\label{eq:morphisms1}
\UU(X) \stackrel{\UU(\iota^\ast)}{\too} \UU(X^H) && \UU(X^H) \stackrel{\UU(\iota_\ast)}{\too} \UU(X)
\end{eqnarray}
in the category $\mathbb{S}\mathrm{ub}^G_0(X)$. Since both these morphisms are $R(G)$-equivariant (see \S\ref{rk:compatibility}), they give rise to well-defined morphisms of ind-objects
\begin{eqnarray}\label{eq:morphisms2}
S^{-1}_H\UU(X) \stackrel{\UU(\iota^\ast)}{\too} S^{-1}_H\UU(X^H) && S^{-1}_H\UU(X^H) \stackrel{\UU(\iota_\ast)}{\too} S^{-1}_H\UU(X)
\end{eqnarray}
in the category $\mathrm{Ind}(\mathbb{S}\mathrm{ub}^G_0(X))$. Under the ring isomorphisms 
$$\End_{\mathbb{S}\mathrm{ub}^G_0(X)}(\UU(X^H))\simeq G_0([X^H/G])\simeq K_0([X^H/G])\,,$$ 
the composition $\UU(\iota^\ast) \circ \UU(\iota_\ast)$ of the morphisms \eqref{eq:morphisms1} (which by definition is given by $[\cO_{X^H}\otimes_X^{\bf L} \cO_{X^H}]$) corresponds to the Grothendieck class $\sum_j (-1)^j [\bigwedge^j(I/I^2)] \in K_0([X^H/G])$, where $I$ stands for the sheaf of ideals associated to the closed immersion $\iota\colon X^H \hookrightarrow X$. Consequently, since $I/I^2=\cN$, the composition $\UU(\iota^\ast) \circ \UU(\iota_\ast)$ of the morphisms \eqref{eq:morphisms2} corresponds to the following morphism of ind-objects
\begin{equation}\label{eq:morph1}
S^{-1}_H\UU(X^H) \stackrel{\lambda_{-1}(\cN)\cdot -}{\too} S^{-1}_H\UU(X^H)\,,
\end{equation}
where $-\cdot -$ stands for the induced action of the Grothendieck group $K_0([X^H/G])$ on the ind-object $S^{-1}_H\UU(X^H)$ (see \S\ref{rk:K0action}).
\begin{lemma}\label{lem:invertible}
The above morphism of ind-objects \eqref{eq:morph1} is invertible.
\end{lemma}
\begin{proof}
Thanks to the Yoneda lemma, it is enough to show that \eqref{eq:morph1} becomes an isomorphism after application of the functor $\Hom_{\mathrm{Ind}(\mathbb{S}\mathrm{ub}^G_0(X))}(S^{-1}_H\UU(X^H),-)$. Recall that $S^{-1}_HR(G)$ can be written as a filtered colimit of free finite $R(G)$-modules. Therefore, it suffices to show that \eqref{eq:morph1} becomes an isomorphism after application of the functor $\Hom_{\mathrm{Ind}(\mathbb{S}\mathrm{ub}^G_0(X))}(\UU(X^H),-)$. By definition of the category $\mathrm{Ind}(\mathbb{S}\mathrm{ub}^G_0(X))$, this latter claim is equivalent to the invertibility of the following homomorphism of abelian groups:
\begin{eqnarray}\label{eq:morph2}
S^{-1}_HK_0([X^H/G]) \too S^{-1}_H K_0([X^H/G]) && \alpha \mapsto \lambda_{-1}(\cN) \cdot \alpha\,.
\end{eqnarray}
As proved by Thomason in \cite[Lem.~3.2]{Thomason}, \eqref{eq:morph2} is indeed invertible.
\end{proof}
Thanks to Lemma \ref{lem:invertible}, we can now consider the following composition:
\begin{equation}\label{eq:composition1}
S^{-1}_H \UU(X^H) \stackrel{(\lambda_{-1}(\cN)\cdot -)^{-1}}{\too} S^{-1}_H\UU(X^H) \stackrel{\UU(\iota_\ast)}{\too} S^{-1}_H\UU(X)\,.
\end{equation}
\begin{proposition}\label{prop:invertible}
The morphism of ind-objects $\UU(\iota^\ast)\colon S^{-1}_H\UU(X) \to S^{-1}_H\UU(X^H)$ is invertible. Moreover, its inverse is given by the above composition \eqref{eq:composition1}.
\end{proposition}
\begin{proof}
Thanks to Lemma \ref{lem:invertible}, $S^{-1}_H\UU(X^H)$ is a direct summand of $S^{-1}_H\UU(X)$. Therefore, using the Yoneda lemma, it is enough to show that $\UU(\iota^\ast)$ becomes an isomorphism after application of the functor $\Hom_{\mathrm{Ind}(\mathbb{S}\mathrm{ub}^G_0(X))}(S^{-1}_H\UU(X),-)$. Moreover, similarly to the proof of Lemma \ref{lem:invertible}, it suffices to show that $\UU(\iota^\ast)$ becomes an isomorphism after application of the functor $\Hom_{\mathrm{Ind}(\mathbb{S}\mathrm{ub}^G_0(X))}(\UU(X),-)$. By definition of the category $\mathrm{Ind}(\mathbb{S}\mathrm{ub}^G_0(X))$, this latter claim is equivalent to the invertibility of the following  homomorphism of abelian groups:
\begin{equation}\label{eq:morph3}
K_0(\iota^\ast)\colon S^{-1}_H K_0([X/G]) \too S^{-1}_H K_0([X^H/G])\,.
\end{equation}
As proved by Thomason in \cite[Thm.~2.1 and Lem.~3.3]{Thomason}, \eqref{eq:morph3} is indeed invertible. %Hence, we conclude that the morphism $\iota_\ast$ is also invertible.

Finally, note that the composition \eqref{eq:composition1} is the right-inverse of $\UU(\iota^\ast)$. Since $\UU(\iota^\ast)$ is invertible, \eqref{eq:composition1} is also the left-inverse of $\UU(\iota^\ast)$.
\end{proof}
We now have the ingredients necessary to conclude the proof of Theorem \ref{thm:localization}. As explained in \S\ref{rk:pullback}, resp. \S\ref{rk:K0action}, resp. \S\ref{rk:compatibility}, the functor $\Psi\colon \mathbb{S}\mathrm{ub}^G_0(X) \to \Hmo_0(k)$ is compatible with pull-backs and push-forwards, resp. with $K_0$-actions, resp. with $R(G)$-actions. Moreover, it extends naturally to the categories of ind-objects:
\begin{equation}\label{eq:ind-functor}
\mathrm{Ind}(\Psi)\colon \mathrm{Ind}(\mathbb{S}\mathrm{ub}^G_0(X)) \too \mathrm{Ind}(\Hmo_0(k))\,.
\end{equation}
Therefore, by combining the preceding functor \eqref{eq:ind-functor} with Proposition \ref{prop:invertible}, we conclude that the morphism of ind-objects $\U(\iota^\ast)\colon S^{-1}_H \U([X/G]) \to S^{-1}_H \U([X^H/G])$ is invertible and that its inverse is given by the following composition:
$$S^{-1}_H \U([X^H/G]) \stackrel{(\lambda_{-1}(\cN)\cdot -)^{-1}}{\too} S^{-1}_H \U([X^H/G]) \stackrel{\U(\iota_\ast)}{\too} S^{-1}_H\U([X/G])\,.
$$
This proves Theorem \ref{thm:localization} in the case of the universal additive invariant. The general case follows now from the equivalence of categories \eqref{eq:equivalence1} and from the fact that every additive functor $\Hmo_0(k) \to \mathrm{D}$ extends naturally to the categories of ind-objects.
\begin{remark}[Generalization]\label{rk:generalization}
Let $\cH$ be a flat quasi-coherent sheaf of algebras over $[X/G]$, \ie a $G$-equivariant flat quasi-coherent sheaf of algebras over $X$. A proof similar to the above one, with $\Psi$ replaced by the functor $\Psi_\cH$ (see Remark \ref{rk:sheaves}), allows us to conclude that Theorem \ref{thm:localization} holds more generally with $S^{-1}_HE([X/G])$ and $S^{-1}_HE([X^H/G])$ replaced by $S^{-1}_HE([X/G];\cH)$ and $S^{-1}_HE([X^H/G];\iota^\ast(\cH))$.
\end{remark}
%---------------------------------------------
\subsection*{Proof of Theorem \ref{thm:roots}}
%---------------------------------------------
Note first that since $R(G)\simeq \bbZ[G^\vee]$ and $G^\vee$ is a finitely generated abelian group, the abelian group $R(G)$ belongs to $\mathrm{Ind}(\mathrm{free}(\bbZ))$. 

Since $\overline{G}:=G\otimes_k \overline{k}$ does not contains a copy of the additive group $\bbG_a$ (in any caracteristic) and the $G$-action on $X^G$ is trivial, Proposition \ref{prop:trivial} and Remark \ref{rk:compactgeneration} yield a Morita equivalence $\perf_\dg(X^G)\otimes \perf_\dg([\bullet/G]) \to \perf_\dg([X^G/G])$. Therefore, using the fact that the universal additive invariant \eqref{eq:universal} is symmetric monoidal, we obtain an induced isomorphism
\begin{equation}\label{eq:isomorphism-2}
\U(X^G) \otimes \U([\bullet/G]) \stackrel{\simeq}{\too} \U([X^G/G])\,.
\end{equation} 
Recall from \S\ref{sec:additive}-\S\ref{sec:action} that the object $\U([\bullet/G])\in \Hmo_0(k)$ carries a canonical commutative monoid structure and that we have natural ring isomorphisms:
\begin{equation}\label{eq:isos}
\Hom_{\Hmo_0(k)}(\U(k),\U([\bullet/G]))\simeq K_0([\bullet/G]) \simeq R(G) \simeq \bbZ[G^\vee]\,.
\end{equation}
Using the characters of $G$, we hence obtain an induced morphism of ind-objects:
\begin{equation}\label{eq:isomorphism-3}
\U(k)\otimes_\bbZ R(G) \too \U([\bullet/G])\,.
\end{equation}
\begin{proposition}
The above morphism of ind-objects \eqref{eq:isomorphism-3} is invertible.
\end{proposition}
\begin{proof}
Note that, thanks to the ring isomorphisms \eqref{eq:isos}, by applying the functor $\Hom_{\mathrm{Ind}(\Hmo_0(k))}(\U(k),-)$ to \eqref{eq:isomorphism-3} we obtain an isomorphism. Hence, in order to prove that \eqref{eq:isomorphism-3} is invertible, it is enough to show that $\U([\bullet/G])$ is isomorphic to a (possibly infinite) direct sum of copies of $\U(k)$. 

Recall that we have an isomorphism $G\simeq \bbG_m^{\times r} \times \mu_{l_1^{\nu_1}} \times \cdots \times \mu_{l_s^{\nu_s}}$ for some prime numbers $l_1, \ldots, l_s$ and non-integers $r, \nu_1, \ldots, \nu_s$. The multiplicative group $k$-scheme $\bbG_m$ is semi-simple. Moreover, the simple $\bbG_m$-representations $V$ are parametrized by the group of characters $\bbG_m^\vee$ and we have $\End_{\bbG_m}(V)\simeq k$. Since, by assumption, $k$ contains the $l^{\mathrm{th}}$ roots of unity, with $l=l_1, \ldots, l_s$, the group $k$-schemes $\mu_{l_1^{\nu_1}}, \ldots,  \mu_{l_v^{\nu_s}}$ are isomorphic to the constant finite group $k$-schemes $C_{l_1^{\nu_1}}, \ldots, C_{l_s^{\nu_s}}$, respectively. In particular, they are semi-simple. Moreover, the simple $\mu_{l^\nu}$-representations $V$ are parametrized by the group of characters $\mu_{l^\nu}^\vee$ and we have $\End_{\mu_{l^\nu}}(V)\simeq k$. These considerations imply that the group $k$-scheme $G$ is also semi-simple and that the dg category $\perf_\dg([\bullet/G])$ is Morita equivalent to the disjoint union $\coprod_{\chi \in G^\vee}k$. Consequently, since $\rep(\coprod_{\chi \in G^\vee}k, \cB)\simeq \prod_{\chi \in G^\vee} \rep(k,\cB)$ for every dg category $\cB$ and the functor $K_0(-)$ preserves arbitrary products, we obtain canonical isomorphisms:
$$ \Hom_{\Hmo_0(k)}(\U([\bullet/G]),\cB)\simeq \prod_{\chi \in G^\vee} \Hom_{\Hmo_0(k)}(\U(k),\cB)\,.$$
This shows not only that the (possibly infinite) direct sum $\bigoplus_{\chi \in G^\vee}\U(k)$ exists in the category $\Hmo_0(k)$, but moreover that $\U([\bullet/G])\simeq \bigoplus_{\chi \in G^\vee}\U(k)$. 
\end{proof}
The above isomorphisms \eqref{eq:isomorphism-2} with \eqref{eq:isomorphism-3} yield an isomorphism of ind-objects $\U(X^G) \otimes_\bbZ R(G) \stackrel{\simeq}{\to} \U([X^G/G])$. Under this latter isomorphism, the natural action of $R(G)$ on the right-hand side corresponds to the canonical $R(G)$-action on $R(G)$. Consequently, we obtain an induced isomorphism of ind-objects:
\begin{equation}\label{eq:isomorphism-5}
\U(X^G) \otimes_\bbZ S^{-1}_G R(G) \stackrel{\simeq}{\too} S^{-1}_G \U([X^G/G])\,.
\end{equation} 
Finally, by combining \eqref{eq:isomorphism-5} with the (inverse of the) isomorphism of ind-objects $S^{-1}_G \U([X/G]) \stackrel{\simeq}{\to} S^{-1}_G \U([X^G/G])$ provided by Theorem \ref{thm:localization}, we obtain an isomorphism of ind-objects $S^{-1}_G \U([X/G])\simeq \U(X^G) \otimes_\bbZ S^{-1}_G R(G)$. This proves Theorem \ref{thm:roots} in the case of the universal additive invariant. The general case follows now from the equivalence of categories \eqref{eq:equivalence1} and from the fact that the natural extension of every additive functor $\Hmo_0(k) \to \mathrm{D}$ to the categories of ind-objects is compatible with the induced action $-\otimes_\bbZ -$ of the category $\mathrm{Ind}(\mathrm{free}(\bbZ))$.
%---------------------------------------------
\subsection*{Proof of Theorem \ref{thm:stabilizers}}
%---------------------------------------------
Let us denote by $\Hmo_0(k)_{1/r}$, resp. by $\mathbb{S}\mathrm{ub}^G_0(X)_{1/r}$, the $\bbZ[1/r]$-linear category obtained by tensoring the abelian groups of morphisms of $\Hmo_0(k)$, resp. of $\mathbb{S}\mathrm{ub}^G_0(X)$, with $\bbZ[1/r]$. In the same vein, let us denote by $\Psi_{1/r}\colon \mathbb{S}\mathrm{ub}^G_0(X)_{1/r} \to \Hmo_0(k)_{1/r}$ the induced $\bbZ[1/r]$-linear functor.

\begin{lemma}\label{lem:finite}
The set of ind-objects $\{\underline{S}^{-1}_{\mu_n} \UU(X)_{1/r}\,|\, \mu_n \in \cC(T) \}$ is finite.
\end{lemma}
\begin{proof}
Note first that the ind-object $\underline{S}^{-1}_{\mu_n} \UU(X)_{1/r}$ is trivial if and only if we have $\mathrm{End}_{\mathrm{Ind}(\mathbb{S}\mathrm{ub}^G_0(X)_{1/r})}(\underline{S}^{-1}_{\mu_n} \UU(X)_{1/r})=0$. Recall that $\underline{S}_{\mu_n}^{-1}R(T)_{1/r}$ can be written as a filtered colimit of free finite $R(T)_{1/r}$-modules. Hence, by definition of the category $\mathrm{Ind}(\mathbb{S}\mathrm{ub}^G_0(X)_{1/r})$, the ind-object $\underline{S}^{-1}_{\mu_n} \UU(X)_{1/r}$ is trivial if and only if we have:
$$ \Hom_{\mathrm{Ind}(\mathbb{S}\mathrm{ub}^G_0(X)_{1/r})}(\UU(X)_{1/r}, \underline{S}^{-1}_{\mu_n} \UU(X)_{1/r})\simeq \underline{S}^{-1}_{\mu_n} K_0([X/T])_{1/r}=0\,.$$
The proof follows now from the fact that the following set of $\bbZ[1/r]$-modules $\{\underline{S}^{-1}_{\mu_n} K_0([X/T])_{1/r}\,|\, \mu_n \in \cC(T)\}$ is finite; see Vezzosi-Vistoli \cite[Prop.~3.4(ii)]{VV}.
\end{proof}
Consider the following canonical morphism of ind-objects:
\begin{equation}\label{eq:morph-loc}
\UU(X)_{1/r} \too \bigoplus_{\mu_n \in \cC(T)}\underline{S}^{-1}_{\mu_n} \UU(X)_{1/r}\,.
\end{equation}
Note that, thanks to Lemma \ref{lem:finite}, the direct sum on the right-hand side is finite.
\begin{proposition}\label{prop:VV}
The above morphism of ind-objects \eqref{eq:morph-loc} is invertible.
\end{proposition}
\begin{proof}
Thanks to the Yoneda lemma, since $\underline{S}_{\mu_n}^{-1}R(T)_{1/r}$ can be written as a filtered colimit of free finite $R(T)_{1/r}$-modules, it suffices to show that \eqref{eq:morph-loc} becomes an isomorphism after application of the functor $\Hom_{\mathrm{Ind}(\mathbb{S}\mathrm{ub}^G_0(X))}(\UU(X)_{1/r},-)$. By definition of the category $\mathrm{Ind}(\mathbb{S}\mathrm{ub}^G_0(X)_{1/r})$, this latter claim is equivalent to the invertibility of the following homomorphism of $\bbZ[1/r]$-modules:
\begin{equation}\label{eq:VV}
K_0([X/T])_{1/r} \too \bigoplus_{\mu_n \in \cC(T)} \underline{S}^{-1}_{\mu_n} K_0([X/T])_{1/r}\,.
\end{equation}
As proved by Vezzosi-Vistoli in \cite[Prop.~3.4(ii)]{VV}, \eqref{eq:VV} is indeed invertible.
\end{proof}
Recall from \S\ref{rk:K0action} and \S\ref{rk:compatibility} that the functor $\Psi$ is compatible with $K_0$-actions and $R(T)$-actions, respectively. The same holds for its $\bbZ[1/r]$-linearization $\Psi_{1/r}$ and for the induced functor $\mathrm{Ind}(\Psi_{1/r})$ between the categories of ind-objects. Therefore, by applying this latter functor to \eqref{eq:morph-loc}, we obtain an isomorphism of ind-objects:
\begin{equation}\label{eq:final1}
\U([X/T])_{1/r} \stackrel{\simeq}{\too} \bigoplus_{\mu_n \in \cC(T)} \underline{S}^{-1}_{\mu_n} \U([X/T])_{1/r}\,.
\end{equation}
\begin{lemma}\label{prop:mult}
For every $\mu_n \in \cC(T)$, we have an isomorphism of ind-objects
\begin{equation}\label{eq:iso-2}
\underline{S}^{-1}_{\mu_n} \U([X/T])_{1/r} \stackrel{\simeq}{\too} \underline{S}^{-1}_{\mu_n} \U([X^{\mu_n}/T])_{1/r}
\end{equation}
induced by pull-back along the closed immersion $X^{\mu_n} \hookrightarrow X$.
\end{lemma}
\begin{proof}
In order to simplify the exposition, let us still denote by $S_{\mu_n}$ the image of the multiplicative set $S_{\mu_n} \subset R(T)$ (see \S\ref{sec:intro}) in the $\bbZ[1/r]$-linearized representation ring $R(T)_{1/r}$. Thanks to Theorem \ref{thm:localization}, we have an isomorphism of ind-objects
\begin{equation}\label{eq:iso-1}
S^{-1}_{\mu_n} \U([X/T])_{1/r} \stackrel{\simeq}{\too} S^{-1}_{\mu_n} \U([X^{\mu_n}/T])_{1/r}
\end{equation}
induced by pull-back along the closed immersion $X^{\mu_n} \hookrightarrow X$. Let $\chi \in T^\vee$ be a character of $T$ whose restriction to $\mu_n$ is non-trivial. As explained by Thomason in the proof of \cite[Lem.~3.6]{VV}, the image of $1- \chi$ under the $\bbZ[1/r]$-algebra homomorphism $R(T)_{1/r} \to \underline{S}^{-1}_{\mu_n} R(T)_{1/r}$ is invertible. Consequently, we obtain an induced $R(T)_{1/r}$-linear homomorphism $S^{-1}_{\mu_n}R(T)_{1/r} \to \underline{S}^{-1}_{\mu_n} R(T)_{1/r}$. Therefore, by applying the functor $-\otimes_{S^{-1}_{\mu_n} R(T)_{1/r}} \underline{S}^{-1}_{\mu_n} R(T)_{1/r}$ to \eqref{eq:iso-1}, we obtain the searched isomorphism of ind-objects \eqref{eq:iso-2}.
\end{proof}
Finally, by combining \eqref{eq:final1} with \eqref{eq:iso-2}, we obtain an isomorphism of ind-objects
$$ \U([X/T])_{1/r} \stackrel{\simeq}{\too} \bigoplus_{\mu_n \in \cC(T)}\underline{S}^{-1}_{\mu_n} \U([X^{\mu_n}/T])_{1/r}$$
induced by pull-back along the closed immersions $X^{\mu_n} \hookrightarrow X$. This proves Theorem \ref{thm:stabilizers} in the case of the universal additive invariant. The general case follows now from the equivalence of categories \eqref{eq:equivalence1} and from the fact that every additive functor $\Hmo_0(k) \to \mathrm{D}$, with values in a $\bbZ[1/r]$-linear category, extends naturally to a $\bbZ[1/r]$-linear functor $\Hmo_0(k)_{1/r} \to \mathrm{D}$ and, consequently, to a $\bbZ[1/r]$-linear functor $\mathrm{Ind}(\Hmo_0(k)_{1/r}) \to \mathrm{Ind}(\mathrm{D})$.

\end{document}

\end{proof}